\documentclass[12pt, reqno]{amsart}
\usepackage{amsmath, amsthm, amscd, amsfonts, amssymb, graphicx, mathtools}

\newtheorem{theorem}{Theorem}[section]
\newtheorem{lemma}[theorem]{Lemma}

\newtheorem{corollary}[theorem]{Corollary}
\theoremstyle{definition}
\newtheorem{definition}[theorem]{Definition}
\newtheorem{example}[theorem]{Example}

\newtheorem{fact}[theorem]{Fact}

\theoremstyle{remark}
\newtheorem{remark}[theorem]{Remark}
\numberwithin{equation}{section}
\newcommand{\reals}{\mathbb{R}}
\newcommand{\complex}{\mathbb{C}}

\newcommand{\field}{\mathbb{F}}

\begin{document}
\setcounter{page}{1}

\title[Schur stability of matrices]{Linear maps on $M_n(\reals)$ preserving Schur stable matrices}
\author[Chandrashekaran]{Chandrashekaran Arumugasamy}

\address{Department of Mathematics\\ 
School of Mathematics and Computer Sciences\\
Central University of Tamil Nadu\\
Neelakudi, Kangalancherry\\
Thiruvarur - 610 101, Tamilnadu, India}
\email{chandrashekaran@cutn.ac.in, chandru1782@gmail.com}

\author[Sachindranath]{Sachindranath Jayaraman}
\address{School of Mathematics\\ 
Indian Institute of Science Education and Research - Thiruvananthapuram\\ 
Maruthamala P.O., Vithura\\ Thiruvananthapuram -- 695 551, Kerala, India}
\email{sachindranathj@gmail.com, sachindranathj@iisertvm.ac.in}



\begin{abstract}
An $n \times n$ matrix $A$ with real entries is said to be Schur stable if all the eigenvalues 
of $A$ are inside the open unit disc. We investigate the structure of linear maps on $M_n(\reals)$ 
that preserve the collection $\mathcal{S}$ of Schur stable matrices. We prove that if $L$ is a linear map 
such that $L(\mathcal{S}) \subseteq \mathcal{S}$, then $\rho(L)$ (the spectral radius of $L$) is atmost $1$ 
and when $L(\mathcal{S}) = \mathcal{S}$, we have $\rho(L) = 1$. In the latter case, the map $L$ preserves 
the spectral radius function and using this, we characterize such maps on both $M_n(\reals)$ as well as on 
$\mathcal{S}^n$.
\end{abstract}
	
\maketitle

\section{Introduction, Notations and Preliminaries}
We work with the field $\reals$ of real numbers throughout. The vector space of all $n \times n$ matrices 
with real entries will be denoted by $M_n(\reals)$. The subspace of real symmetric matrices will be denoted 
by $\mathcal{S}^n$. For $A \in M_n(\reals)$, the spectrum of $A$ will be denoted by $\sigma(A)$ and the 
spectral radius of $A$, denoted by $\rho(A)$, is the number $\displaystyle \max_{\lambda \in \sigma(A)} |\lambda|$. 
We denote the operator norm or the spectral norm of $A$ by $||A||:= \displaystyle \sup_{||x||_2 = 1} ||Ax||_2$, 
where $||.||_2$ denotes the Euclidean norm of an element of $\reals^n$. It is well known that the the operator norm 
as well as the spectral radius are unitarily invariant on $M_n(\reals)$. Moreover, unlike the Frobenius norm, 
the operator norm is not induced by an inner product. One may refer to either \cite{hj-1} 
or \cite{z2} for results on matrix theory.

\begin{definition}\label{defn-1}
An $n \times n$ matrix $A$ with real or complex entries is said to be Schur stable if all the eigenvalues 
of $A$ are inside the open unit disc. 
\end{definition}

The following result concerning stability of the dynamical system $x_{k+1} = Ax_k, \ A \in M_n(\reals)$ 
is well known. The matrix $A$ is Schur stable if and only if for every initial condition $x_0 \in \reals^n, \ 
\displaystyle \lim_{k \to \infty} x_k = 0$. This is also equivalent to $\displaystyle \lim_{k \to \infty} A^k = 0$. 
The proof of the above result can be found in Chapter 2 of \cite{hj-2}. Schur stability of $A$ can also be 
formulated in terms of positive definite solutions $X$ and $R$ to the discrete-time Lyapunov equation:
\begin{equation*}
X - A^tXA - R = 0. 
\end{equation*}

\begin{theorem}\label{thm-folklore-1}
Let $A \in M_n(\reals)$. Then, the following statements are equivalent:
\begin{enumerate}
\item $A$ is Schur stable.
\item For each symmetric positive definite matrix $R$, there exists a positive definite matrix $X$ such 
that $X - A^tXA - R = 0$.
\item For each symmetric matrix $Q$, the following optimization problem has a solution: Find a symmetric 
positive definite matrix $X$ such that $Y:= X - A^tXA + Q$ is symmetric positive definite and $trace(YX) = 0$.
\end{enumerate}
\end{theorem}

A comprehensive treatment of Schur stability as well as two of its variants, namely, Schur $\mathbb{D}$ 
stability and diagonal stability, can be found in Chapter 2 of the book by Kaskurewicz and Bhaya \cite{kb}. 
The problem in the third statement of Theorem \ref{thm-folklore-1} is called the semidefinite linear 
complementarity problem and originally considered in \cite{ksh} as an optimization problem. The proof of 
Theorem \ref{thm-folklore-1}, along with other equivalent statements can be found in Theorem 11, \cite{gp}.

For a field $\field$ and the set $M_{m,n}(\field)$ of $m \times n$ matrices over $\field$, 
a linear preserver $\phi$ is a linear map $\phi : M_{m,n}(\field) \longrightarrow M_{m,n}(\field)$ 
that preserves a certain property or a relation. Most such maps are of the form $\phi(A) = MAN$ for some 
invertible matrices $M$ and $N$ of orders $m \times m$ and $n \times n$, respectively, or $m = n$ and 
$\phi(A) = MA^tN$ for some invertible matrices $M$ and $N$ of order $n \times n$. Both these maps are called 
standard maps. The first such problem was studied by Frobenius, who proved that any determinant preserver 
is of the form $MAN$ or $MA^tN$ with $det(MN) = 1$. Other properties of matrices such as rank, inertia, 
invertibility, functions of eigenvalues and so on, were investigated later on. For instance, rank preservers on 
the space of complex as well as real matrices are of the above form. For general techniques on linear preserver 
problems, one may refer to \cite{lp} or \cite{lt} and the references cited therein. There are two types of preserver 
problems one usually considers. Given a subset $\mathcal{S}$ of $M_{m,n}(\field)$, what are the linear maps $\phi$ on 
$M_{m,n}(\field)$ such that $\phi(\mathcal{S}) \subseteq \mathcal{S}$. Such a map is called an into preserver. 
The other one is to characterize those linear maps $\phi$ on $M_{m,n}(\field)$ such that $\phi(\mathcal{S}) = 
\mathcal{S}$. Maps of this type are called onto preservers. It is known that when $M_n(\reals)$ has a basis 
consisting of elements of $\mathcal{S}$, then onto preservers are precisely those into preservers that are 
invertible and the inverse being an into preserver (see \cite{d} for details).  

Let $\mathcal{S}$ be the collection of real matrices that are Schur stable. Our objective is to study linear 
preservers of Schur stability on the vector spaces $M_n(\reals)$ and $\mathcal{S}^n$. We prove that if $L$ is an 
invertible linear map on $M_n(\reals)$ that preserves $\mathcal{S}$, then $\rho(L) \leq 1$, although the converse 
is not true (Theorem \ref{thm-3}). Equality holds in the case of an onto preserver (Theorem \ref{thm-4}). In this 
case, the map $L$ preserves the spectral radius function as well. We completely characterize onto preservers on 
$M_n(\reals)$ as well as on $\mathcal{S}^n$ (Corollary \ref{cor-2} and Remarks \ref{rem-2}). Instances when 
maps with $\rho(L) \leq 1$ preserve $\mathcal{S}$ are also discussed (Theorem \ref{thm-2}) and examples are 
presented to illustrate or substantiate our results.

Before proceeding further, it must be pointed out that linear maps on $M_n(\complex)$ that preserve regional 
eigenvalue locations was studied by Johnson et al \cite{jlrsp}. Let $H(r,s,t)$ denote the collection of complex 
matrices with $r$ eigenvalues inside the unit disc, $s$ eigenvalues outside the unit disc and $t$ eigenvalues on 
the unit circle. Theorem $4.1$ of \cite{jlrsp} says the following.

\begin{theorem} \label{thm-jlrsp}
Let $T$ be an invertible linear map on $M_n(\complex)$ such that $T(H(r,s,t)) \subseteq H(r,s,t)$. Assume that 
$r = n, s = t = 0$. Then, there exists $\alpha, \beta \in \complex$ such that $\beta \neq 0, \ \alpha n + \beta \neq 0, 
\ (n-1) |\alpha| + |\beta + \alpha| \leq 1$ and an invertible $S \in M_n(\complex)$ such that 
$T(A) = \alpha (trace A) I + \beta S^{-1}AS$ for all $A \in M_n(\complex)$\\
or \\
$T(A) = \alpha (trace A) I + \beta S^{-1}A^tS$ for all $A \in M_n(\complex)$.\\
Conversely, if $T$ is given by either of the above representations with at $\beta \neq 0, \ \alpha n + \beta \neq 0, 
\ (n-1) |\alpha| + |\beta + \alpha| \leq 1$, then $T$ is nonsingular and $T(H(n,0,0)) \subseteq H(n,0,0)$.
\end{theorem}

As pointed in the remark after Theorem $4.1$ in \cite{jlrsp}, nonsingularity of the map $T$ is essential. We 
therefore make this assumption henceforth in this manuscript. Since our motivation to study linear preservers of 
Schur stable matrices stems from Theorem \ref{thm-folklore-1}, we focus on the case $M_n(\reals)$ alone. We end this 
section by pointing out that results in \cite{jlrsp} do not verbatim carry over the real field. However, some of our 
results are also valid for mpas on $M_n(\complex)$, too as the proofs would indicate.

\section{Main Results}
We present the main results in this section. The following are some well known facts that will be used in this paper.

\begin{fact}\label{fact-0}
The following hold for $A, B \in M_n(\reals)$.
\begin{enumerate}
\item[1] $\rho(AB) = \rho(BA)$.
\item[2] If $U \in M_n(\reals)$ is orthogonal, then $\rho(UAU^t) = \rho(A)$.
\item[3] If $U \in M_n(\reals)$ is orthogonal and if $A \in M_n(\reals)$ is Schur stable, so is $UAU^t$.
\item[4] If $A$ and $B$ commute, then $\rho(AB) \leq \rho(A) \rho(B)$.
\end{enumerate}
\end{fact}

The following lemma gives an easy criterion to check Schur stability of $2 \times 2$ matrices.

\begin{lemma}\label{lem-0}
(Lemma 2.7.16, \cite{kb}) $A \in M_2(\reals)$ is Schur stable if and only if $|trace(A)| < 1 + det(A) \ \mbox{and} \ 
|det(A)| < 1$. 
\end{lemma}

We begin with the following definition before proving our results. 

\begin{definition}\label{defn-2}
$A \in M_n(\field)$ (where $\field$ is either $\reals$ or $\complex$) is called normaloid if $\rho(A) = ||A||$ and 
spectraloid if $\rho(A) = w(A)$, where $w(A):= \displaystyle \sup_{||x|| = 1} \langle Ax,x \rangle$.
\end{definition}

It is clear that normal matrices are normaloid and the class of normaloid matrices is strictly contained in the 
class of spectraloid matrices. Before proving our main results, we observe the following. 

\begin{fact}\label{fact-1}
Let $L$ be a linear map such that $\rho(A) = \rho(L(A))$ for all $A \in M_n(\reals)$. Then, $L$ preserves Schur 
stability.
\end{fact}

The converse of the above statement holds when $L$ is an onto preserver of Schur stability. This will be proved later 
on in Theorem \ref{thm-4}. The following lemma is a useful observation that will be used later on.

\begin{lemma}\label{lem--1}
Both $M_n(\reals)$ as well as $\mathcal{S}^n$ have bases consisting of Schur stable matrices. 
\end{lemma}

We state below the following well known theorem. The proof may be found in Theorem 1 and the Corollary following it in 
\cite{m}.

\begin{lemma}\label{lem--2}
A linear map $L$ on $M_2(\complex)$ preserves the trace and determinant if and only if it preserves the spectrum.
\end{lemma}

The following is a result on linear preservers of Schur stable matrices that are normaloid. 

\begin{theorem}\label{thm-1}
Suppose $M$ and $N$ are commuting symmetric (normal) matrices such that $\rho(MN) \leq 1$. Then $MAN$ is Schur 
stable whenever $A$ is normaloid and Schur stable.
\end{theorem}

\begin{proof}
Note that if $M$ and $N$ are diagonal matrices such that $\rho(MN) \leq 1$, then $MAN$ is Schur stable whenever 
$A$ is normaloid and Schur stable. This easily follows as $\rho(MAN) = \rho(NMA) \leq ||NMA|| \leq ||NM|| ~ ||A|| 
= \rho(MN) ||A||$. Suppose $M$ and $N$ are commuting symmetric (normal) matrices. Then, there exists an orthogonal 
matrix $U$ and diagonal matrices $D_1, D_2$ such that $M = UD_1U^t, N = UD_2U^t$. Notice that $\rho(D_1D_2) \leq 1$. 
Now, $MAN = UD_1U^tAUD_2U^t$. The proof follows from the diagonal case discussed above and the facts that spectral 
radius and norm are unitarily invariant.
\end{proof}

\begin{remark}\label{rem-0}
In particular, if $X$ is a normal matrix such that $\rho(XX^t) \leq 1$, the map $A \mapsto XAX^t$ will preserve Schur 
stability. The noncommuting case can also be done now. Consider the map $A \mapsto MAN$ with $\rho(MN) \leq 1$. Note 
that $\rho(MAN) = \rho(NMA) \leq ||NM|| ||A||$. Assume that $A$ is normaloid and Schur stable. It then easily follows 
that if $NM$ is symmetric or normal or normaloid, then the above map preserves Schur stability.
\end{remark}

\begin{theorem}\label{thm-2}
Let $L$ be a normal (symmetric) linear map on $M_n(\reals)$ such that $L = U \tilde{L} U^t$ for some orthogonal 
transformation $U$ on $M_n(\reals)$ and some diagonal transformation $\tilde{L}$ on $M_n(\reals)$. If 
$\rho(L) \leq 1$, then $L(A)$ is Schur stable whenever $A$ is normaloid and Schur stable.
\end{theorem}

\begin{proof}
Let us first observe that if $L$ is normal and $\rho(L) \leq 1$, then $||L||$ (the operator norm of $L$) is 
less than or equal to $1$. Let $A$ be normaloid and Schur stable. Then, $\rho(L(A)) \leq ||L(A)|| \leq 
||L|| ||A|| = \rho(L) \rho(A) < 1$.
\end{proof}

In $\mathcal{S}^n$, the class of normaloid matrices and spectraloid matrices coincide. Hence the following 
result follows easily. 

\begin{theorem}\label{thm-2.1}
If $L$ is a normal map on $\mathcal{S}^n$ with $\rho(L) \leq 1$, then $L$ is an into preserver of Schur stability. 
\end{theorem}

We now prove that if a linear map $L$ preserves Schur stability, then $\rho(L) \leq 1$. Consequently, 
if the map $A \mapsto MAN$ preserves Schur stability, then $\rho(MN) \leq 1$. 

\begin{theorem}\label{thm-3}
Let $L$ be a linear map on $M_n(\reals)$. If $L(A)$ is Schur stable whenever $A$ is Schur stable, then 
$\rho(L) \leq 1$.
\end{theorem}

\begin{proof}
Suppose $\rho(L(A)) > \alpha > \rho(A)$ for some $A$. Then $\alpha > 0$ and the matrix $B = (1/\alpha)A$ is 
Schur stable. But the matrix $L(B)$ is not Schur stable. Therefore, $\rho(L(A)) \leq \rho(A)$ for each 
$A \in M_n(\reals)$. It now follows that $\rho(L) \leq 1$. 
\end{proof}

\begin{remark} \label{rem-1}
The above result is not surprising as the condition obtained in Theorem \ref{thm-jlrsp} for the complex case 
is very similar. In fact, the above result holds for linear maps on $M_n(\complex)$, too. The examples discussed 
later on illustrate various possible cases as well as substantiate our results. 
\end{remark}

We now move on to the case of onto preservers of Schur stability. The following lemma is a consequence of 
Theorem \ref{thm-3}.

\begin{lemma}\label{lem-2}
Let an invertible linear map $L$ on $M_n(\reals)$ be an onto preserver of Schur stability. Then, $\rho(L) = 1$.  
\end{lemma}

\begin{proof}
Recall that $M_n(\reals)$ has a basis consisting of Schur stable matrices. Therefore, onto preservers of $\mathcal{S}$ 
are those into preservers that are invertible and the inverse being an into preserver. We therefore have 
$\rho(L) \leq 1$ as well as $\rho(L^{-1}) \leq 1$. But $\rho(L) \rho(L^{-1}) \geq 1$. It follows that $\rho(L) = 1$.
\end{proof}

In particular, if the map $L(X) = AXA^t$ on $M_n(\reals)$ is an onto preserver of Schur stability, then all the 
eigenvalues of $A$ have absolute value $1$.

\begin{lemma}\label{lem-3}
Consider the map $L(X) = AXA^t$ on $M_n(\reals)$. If $L$ is an onto preserver of Schur stability, then all the 
eigenvalues of $A$ have absolute value $1$.  
\end{lemma}

\begin{proof}
Let the eigenvalues of $A$ be $\lambda_1, \ldots, \lambda_n$, counting multiplicities. Then, the eigenvalues of $L$ 
are of the form $\lambda_i \lambda_j$. Therefore, $\rho(L) = \rho(A)^2$. It now follows that if $L$ is an onto 
preserver of Schur stability, then $\rho(A) = 1$. Since the map $L$ is an onto preserver, $L^{-1}$ is an into 
preserver and so $\rho(L^{-1}) \leq 1$. Note that $L^{-1}(X) = A^{-1}X(A^t)^{-1}$. A similar calculation as above 
implies that $1 \geq \rho(A^{-1})^2$ and so $\rho(A^{-1}) \leq 1$. However, $\rho(A^{-1}) \geq 1$. Therefore, 
$\rho(A) = \rho(A^{-1}) = 1$, which implies that all the eigenvalues of $A$ lie on the unit circle.
\end{proof}

We now prove that if $L$ is an onto preserver of Schur stability, then $L$ preserves spectral radius.

\begin{theorem}\label{thm-4}
If $L$ is an onto preserver of Schur stability, then $L$ preserves spectral radius (that is, $\rho(A) = \rho(L(A))$ 
for all $A \in M_n(\reals)$). 
\end{theorem}

\begin{proof}	
Since the set $\mathcal{S}$ contains a basis, onto preservers are precisely those into preservers that are invertible 
and whose inverse is also an into preserver. Recall that $\rho(\alpha A) = |\alpha| \rho(A)$. It is enough to prove 
that $\rho(L(A)) = 1$, whenever $\rho(A) = 1$. Let $A \in M_n(\mathbb R$) be such that $\rho(A) = 1$. Then $\alpha A$ 
is Schur stable whenever $|\alpha| < 1$. We claim that $\rho(L(A))$ cannot be less than $1$. Suppose $\rho(L(A)) < 1$. 
Then there exists $\beta > 1$ such that $\beta \rho(L(A)) = \rho(\beta L(A)) = \rho(L(\beta A)) < 1$. Observe that 
$\rho(\beta A) = \beta > 1$. This contradicts the fact that $L^{-1}$ is an into preserver of Schur stability, for 
$L^{-1}(L(\beta A)) = \beta A$ is not Schur stable whereas $L(\beta A)$ is Schur stable. That means $\rho(L(A)) \geq 1$. 
We now claim that $\rho(L(A))$ cannot be bigger than $1$ as well. Suppose $\rho(L(A)) > 1$. Then there exists 
$\gamma \in (0,1)$ such that $\gamma \rho(L(A)) = \rho(L(\gamma A)) > 1$. Then by a similar argument and noting 
$\gamma A$ is Schur stable, we arive at a contradiction to the fact that $L$ preserves Schur stability. 
Thus $\rho(L(A)) = 1$.
\end{proof}

We have the following corollaries of Theorems \ref{thm-3} and \ref{thm-4}.

\begin{corollary}\label{cor-1}
If $L$ is a preserver of Schur stability, then $L$ preserves nilpotent matrices.
\end{corollary}

\begin{proof}
The proof follows as $L$ satisfies $\rho(L(A)) \leq \rho(A)$ for each $A$.
\end{proof}

\begin{corollary}\label{cor-2}
A map $L$ defined on $\mathcal{S}^n$ is an onto preserver of Schur stability if and only if $L(A) = c RAR^t$ for some 
orthogonal matrix $R$ and $c \in \{1,-1\}$. 
\end{corollary}

\begin{proof}
For the \textit{only if} part, observe that from Theorem \ref{thm-4}, we know that $\rho(A) = \rho(L(A))$ for all 
$A \in \mathcal{S}^n$. It is easy to prove that for any symmetric matrix, $\rho(A) = ||A||$ (the operator norm of $A$). 
Thus, $L$ is an isometry with respect to the operator norm. Recall that the operator norm is unitary congruence invariant; 
that is, for any $A \in \mathcal{S}^n$ and any orthogonal matrix $U$, we have $||UAU^t|| = ||A||$. Moreover, it is not 
induced by an inner product. It follows from (a suitable modification of) Theorem 2 of \cite{lt-2} that $L(A) = c RAR^t$ 
for some orthogonal matrix $R$ and $c \in \{1,-1\}$.

For the \textit{if} part, suppose $L(A) = c RAR^t$ 
for some orthogonal matrix $R$ and $c \in \{1,-1\}$. Then, $\rho(c RAR^t) = \rho(RAR^t) = \rho(R^tRA) 
\leq ||R^tRA|| = ||A|| = \rho(A)$. This proves that $L$ is an into preserver of Schur stability. Note that 
$L^{-1}(A) = c R^tAR$. A similar calculation proves that $L^{-1}$ will be an into preserver of Schur stability. 
This shows that $L$ is an onto preserver of Schur stability, as $\mathcal{S}^n$ has a basis consisting of Schur 
stable matrices.
\end{proof}

\begin{remark}\label{rem-2}
A few remarks are in order.
\
\begin{enumerate}
\item It is possible for a Schur stability preserver $L$ to have a nilpotent eigenvector. For example, the invertible 
linear map $L: M_2(\reals) \rightarrow M_2(\reals)$ defined by $A = \begin{bmatrix*}[r]
                                                          a & b\\
                                                          c & d
                                                         \end{bmatrix*} \mapsto \begin{bmatrix*}[r]
                                                                                a & 2b\\
                                                                                \frac{c}{2} & d
                                                                               \end{bmatrix*}$ preserves both trace 
as well as determinant and so will preserve the spectrum. Consequently, it will also preserve Schur stability. It is 
easily seen that $L(N) = 2N$, where $N$ is the standard nilpotent matrix. 
\item Theorem \ref{thm-4} is also true for maps on $M_n(\complex)$, as $M_n(\complex)$ has a 
basis consisting of Schur stable matrices. Consequently, $L(A) = c TAT^{-1}$ or $L(A) = c TA^*T^{-1}$ 
for some invertible matrix $T$ for all $A \in M_n(\complex)$ and a nonzero complex number $c$ with $|c| = 1$. 
For a proof, one can refer to \cite{bs}. 
\item Notice that the map $A \mapsto c TAT^{-1}$ with $|c| = 1$, is always an onto preserver of Schur stability. 
Conversely, we know from Theorem \ref{thm-4} that onto preservers of Schur stability preserve the spectral radius and 
is therefore of the form $A \mapsto c TAT^{-1}, |c| = 1$ for some invertible matrix $T$. We have therefore characterized 
onto preservers of Schur stability on $M_n(\reals)$.
\item If $L$ is an onto preserver of nilpotent matrices on $M_n(\reals)$, 
then $L(A) = c TAT^{-1} - \frac{c}{n} trace(a_{ij} I + \frac{1}{n} \phi(trace (a_{ij} I))$ or $L(A) = c TA^tT^{-1} - 
\frac{c}{n} trace(a_{ij}^t I + \frac{1}{n} \phi(trace (a_{ij}^t I))$ for some nonzero scalar $c$, an 
invertible matrix $T$ and an additive map $\phi: \reals I \longrightarrow M_n(\reals)$.
\item If we are restricting to only the subspace of trace zero matrices, $sl_n(\reals)$, then any onto preserver 
of nilpotent matrices will be of the form $L(A) = cTAT^{-1}$ or $L(A) = cTA^tT^{-1}$, with $0 \neq c \in \reals$. 
One can refer to Theorem 2.3 and Corollary 2.5 of \cite{bh} for a proof of the above statements. Therefore, when 
$c = |c| = 1$, any onto preserver of nilpotent matrices will always be an onto preserver of Schur stability in 
$sl_n(\reals)$.
\end{enumerate}
\end{remark}

We have thus characterized onto preservers of Schur stability both on $\mathcal{S}^n$ as well as on $M_n(\reals)$. 
Over $\mathcal{S}^n$, such maps will be of the form $L(A) = c RAR^t$ for some orthogonal matrix $R$ and $c \in \{1,-1\}$, 
whereas over $M_n(\reals)$, they will be of the form $L(A) = c TAT^{-1}$ or $L(A) = c TA^tT^{-1}$ for some invertible 
matrix $T$ and $|c| = 1$. Notice that for the \textit{if} part of Corollary \ref{cor-2}, it suffices to assume that $L$ 
preserves the spectral radius function. Then, it preserves the operator norm $||.||$, which is unitarily invariant. 
It will then follow from (a suitable modification of) Theorem 2 of \cite{lt-2} that $L(A) =  c UAU^t$ for some real 
orthogonal matrix $U$ and $c \in \{1,-1\}$. 

From the Theorems and Examples presented thus far, we know that an invertible map $L$ on $M_n(\reals)$ 
with $\rho(L) = 1$ need not preserve Schur stable matrices. Therefore, the natural question is whether additional 
conditions on $L$ will ensure such a result. We do not know the complete answer to this question. The following can 
be proved: If $L$ is an invertible normal map on $\mathcal{S}^n$ such that $\rho(L) = \rho(L^{-1}) = 1$, then 
$L(A) = \pm RAR^t$ for some orthogonal matrix $R$ and so will be an onto preserver of Schur stability. This is 
because the conditions on $L$ will imply that $L$ is an isometry. A similar statement can be made for maps on 
$M_n(\reals)$, except the map will be of the form $UAV$ for orthogonal matrices $U$ and $V$ (see Theorem 4.1, 
\cite{lt-1} for a proof).

\subsection*{Examples}\hspace*{\fill} \\

We now present examples that illustrate possible cases and also substantiate our results.

\begin{example}\label{eg-2.0}
Consider $M = \begin{bmatrix*}[r]
               0.5 & 0 & 10\\
               0 & 0.5 & 0\\
               0 & 5 & 0.5
              \end{bmatrix*}$ and $N = I$. Then, $MN$ is invertible and it can be verified that $\rho(MN) \leq 1$. 
Now consider the matrix $A = \begin{bmatrix*}[r]
                              0.5 & 0 & 0\\
                              0 & 0 & 1\\
                              0 & 0 & 0
                             \end{bmatrix*}$. It can be easily computed that $||A|| = 1$ and $w(A) = \rho(A) = 1/2$. 
Therefore $A$ is spectraloid but not normaloid. One now easily varifies that $MAN$ is not Schur stable. 
\end{example}

The above example shows that the class of normaloid Schur stable stable matrices is at present the maximal class for 
which we have an into preserver. Recall that when $n = 2$, a matrix $A$ is spectraloid if and only if it is normal 
(and hence normaloid). The following example shows thtat even if $MN$ is Schur stable, the map $A \mapsto MAN$ 
need not preserve Schur stability.

\begin{example}\label{eg-2.1}
Let $A = \begin{bmatrix*}[r]
          1.17258 & 1.35575\\
          -0.94256 & -0.39761
         \end{bmatrix*}, \ M = \begin{bmatrix*}[r]
			   0.79323 & 0\\
			   0 & -0.24866
			   \end{bmatrix*}$ and $N = I$. Then, $A$ and $M$ are Schur stable, $\rho(MN) = \rho(M) < 1$. 
However, $MAN$ is not Schur stable, as its spectral radius is $1.1626$. Note that $A$ is not normaloid, as its 
norm is $2.0245$, whereas its spectral radius is $0.90091$. 
\end{example}

There are also matrices $M$ and $N$ with one or both being non-symmetric, $\rho(MN) \leq 1$, but the map 
$A \mapsto MAN$ fails to preserve Schur stability. The followig example illustrates this.

\begin{example}\label{eg-2.2}
Let $M = \begin{bmatrix*}[r]
          1 & -1\\
          0 & -1
         \end{bmatrix*}, N = \begin{bmatrix*}[r]
                             1 & 0\\
                             0 & 1
                             \end{bmatrix*}$. Then, $M$ is not symmetric, whereas $N$ is. $MN = M$. 
Therefore, $\rho(MN) = 1$. Let $A = \begin{bmatrix*}[r]
                                     0.5 & 0\\
                                     10 & -0.5
                                    \end{bmatrix*}$. $A$ is Schur stable. However, $MAN = \begin{bmatrix*}[r]
											  49.5 & 2.5\\
											  -10 & 0.5
											  \end{bmatrix*}$ 
is not Schur stable, as $\rho(MAN) = \rho(MA) = 48.99$. Notice that $M$ is not diagonalizable, whereas $N$ is.
\end{example}

As the next example illustrates, even when both $M$ and $N$ are symmetric with $\rho(MN) \leq 1$, it may happen 
that the map $A \mapsto MAN$ fails to preserve Schur stability.

\begin{example}\label{eg-2.3}
Let $M = \begin{bmatrix*}[r]
          -1 & -0.5\\
          0.5 & -1
         \end{bmatrix*}, N = (2/3)I$. Both $M$ and $N$ are symmetric and $\rho(MN) = (2/3)\rho(M) = 1$. 
Let $A = \begin{bmatrix*}[r]
         0.5 & 100\\
          0 & -0.5
         \end{bmatrix*}$, which is Schur stable. However, $MAN = (2/3)MA = (2/3) \begin{bmatrix*}[r]
									        -0.5 & -100.25\\
									         0.25 & 50.5
										\end{bmatrix*}$ 
is not Schur stable, as its spectral radius is $33.33$. 
\end{example}

\begin{example}\label{eg-2.4}
Consider the following map $L$ on $M_2(\reals)$ defined by $\begin{bmatrix*}[r]
                                                                   a & b\\
                                                                   c & d
                                                            \end{bmatrix*} \mapsto \begin{bmatrix*}[r]
										   a & b\\
										   2d & -2c
										   \end{bmatrix*}$. Take 
$L^{\prime} = (1/4)L$, which is clearly normal. Moreover, $\rho(L^{\prime}) = 1/2 < 1$. Now consider the 
matrix $A = \begin{bmatrix*}[r]
             0.75 & 5\\
             0 & -0.75
            \end{bmatrix*}$, which is Schur stable. $L^{\prime}(A) = \begin{bmatrix*}[r]
                                                                     0.75 & 5\\
                                                                     -1.5 & 0
                                                                     \end{bmatrix*}$, which can be easily seen to 
have spectral radius bigger than $1$. Note that although the map $L^{\prime}$ is normal and has spectral radius less 
than $1$, the matrix $A$ is not normaloid (but Schur stable though).
\end{example}

It follows from the above examples that linear preservers of Schur stable matrices is properly contained in the 
class of those linear maps whose spectral radius is at most $1$.

The following are examples of singular linear maps on $M_2(\reals)$ which do not preserve Schur stability. The 
first one is non-diagonalizable, but contractive and has spectral radius less than $1$, whereas the second one is 
neither diagonalizable nor contractive, but has spectral radius less than $1$.

\begin{example}\label{eg-2.5}
Consider the linear map $L$ on $M_2(\reals)$ defined by $\begin{bmatrix}
                                                              a & b\\
                                                              c & d
                                                              \end{bmatrix} \mapsto \begin{bmatrix}
                                                                                     b & c\\
                                                                                     0 & 0
                                                                                    \end{bmatrix}$. 
In this case, $\rho(L) = 0, \ ||L|| = 1$. However, $L$ is neither diagonalizable nor does it preserve 
Schur stability.
\end{example}

\begin{example}\label{eg-2.6}
Let $L : M_2(\reals) \longrightarrow M_2(\reals)$ be given by $\begin{bmatrix}
                                                                a & b\\
                                                                c & d
                                                                \end{bmatrix} \mapsto \begin{bmatrix}
										       b & 2c\\
										       0 & 0
										       \end{bmatrix}$. 
Then $\rho(L) = 0 < 1$. Consider the matrix $A = \begin{bmatrix}
                                                  1/2 & 2\\
                                                  0 & 1/2
                                                 \end{bmatrix}$, which is Schur stable. 
$L(A) = \begin{bmatrix}
         2 & 0\\
         0 & 0
        \end{bmatrix}$ is not Schur stable. The operator norm of above linear map $L$ is $2$ and moreover 
$L$ is not diagonalizable. This implies there are linear maps that have spectral radius less than $1$ but are neither 
diagonalizable nor contractions and do not preserve Schur stability.
\end{example}

It follows from the above examples that linear preservers of Schur stable matrices is properly contained in the 
class of those linear maps whose spectral radius is at most $1$.

\section{Concluding Remarks}
We have studied in this paper the real version of a specific eigenvalue location preserver problem, namely, 
linear preservers of Schur stability on $M_n(\reals)$. Come of our results are also true for maps on $M_n(\complex)$ 
as well, as the proofs indicate. As pointed out in \cite{jlrsp}, eigenvalue location 
preserver problems for real matrices is considerably hard and may require different methods than those that 
are commonly used in the complex case, as the real field is not algebraically closed. 
Although a characterization of into preservers of Schur stable matrices is elusive to us at this point, 
we have proved that if a map $L$ preserves Schur stable matrices, then its spectral radius is necessarily at most 
$1$ in the case of into preservers and equal to $1$ in the case of onto preservers. Onto preservers of Schur stability 
on $M_n(\reals)$ as well on $\mathcal{S}^n$ are completely characterized. Several other examples are also provided 
to substantiate our results. 

\vspace{.25cm}
\noindent
\textbf{Acknowledgements:} The first author thanks the SERB, Government of India, for partial support in the form of a 
grant (Grant No. ECR/2017/000078).The authors also thank the Kerala School of Mathematics, Kozhikode, India, for providing 
necessary support to carry out parts of this work. 

\bibliographystyle{amsplain}
\begin{center}

\end{center}
\end{document}